\documentclass{amsart}

\usepackage[latin1]{inputenc}
\usepackage[T1]{fontenc}
\usepackage{amsmath,amsfonts,amssymb,amsthm}
\usepackage{graphicx}
\usepackage[all]{xy}
\usepackage{microtype}
\usepackage{mathrsfs}

\theoremstyle{plain}
\newtheorem{theorem}{Theorem}[section]
\newtheorem{lemma}[theorem]{Lemma}

\newtheorem{cor}[theorem]{Corollary}

\theoremstyle{definition}

\theoremstyle{remark}
\newtheorem{rmk}[theorem]{Remark}

\newtheorem*{ack}{Acknowledgements}

\renewcommand{\tilde}[1]{\widetilde{#1}}
\newcommand{\set}[1]{{\left\{#1\right\}}}
\newcommand{\pa}[1]{{\left(#1\right)}}

\newcommand{\abs}[1]{{\left|#1\right|}}
\newcommand{\ra}{\rightarrow}
\newcommand{\lra}{\longrightarrow}
\newcommand{\sheaf}[2]{\mathcal{#1}_{#2}}
\newcommand{\fascio}[1]{\sheaf{O}{#1}}

\newcommand{\prym}{\mathscr{P}}

\newcommand{\m}[1]{\mathcal{#1}}

\newcommand{\C}{\mathbb{C}}

\newcommand{\Z}{\mathbb{Z}}

\DeclareMathOperator{\sym}{Sym}

\DeclareMathOperator{\pic}{Pic}

\DeclareMathOperator{\Id}{Id}

\title{Prym varieties of double coverings of elliptic curves}
\author{Valeria Ornella Marcucci}
\address{Dipartimento di Matematica ``F. Casorati''\\
	Universit\`a di Pavia\\
	via Ferrata 1, 27100 Pavia, Italy}
	\email{valeria.marcucci@unipv.it}
\author{Juan Carlos Naranjo}
\address{
Departament d'\`Algebra i Geometria \\
Facultat de Matem\`atiques \\
Universitat de Barcelona \\
Gran Via 585 \\
08007 Barcelona, Spain }
\email{jcnaranjo@ub.edu}

\thanks{J.C.Naranjo has been partially supported by the Proyecto de Investigaci\'on MTM2009-14163-C02-01.}
\thanks{V. Marcucci has been partially supported by 1) FAR 2010 (PV) \emph{"Variet\`a algebriche, calcolo algebrico, grafi orientati e topologici"}; 2) INdAM (GNSAGA) 3) PRIN 2009 \emph{``Spazi di moduli e teorie di Lie''}
}


\newcommand{\query}[1]{\marginpar{
\vskip-\baselineskip 
\raggedright\footnotesize                                   
\itshape\hrule\smallskip#1\par\smallskip\hrule}}            
\newcommand{\removequeries}{\renewcommand{\query}[1]{}}    
\removequeries

\begin{document}

\begin{abstract}
We prove the generic injectivity of the Prym map $\prym\colon \m{R}_{1,r} \ra \m{A}_{\frac{r}{2}}^\delta$ sending a double covering of an elliptic curve ramified at $r\geq 6$ points to its polarized Prym variety. For $r=6$ the map is birational and both $\m{R}_{1,6}$ and $\m{A}_{3}^\delta$ are unirational.
\end{abstract}

\maketitle

\section{Introduction}

The classical Prym map $\prym\colon \m{R}_{g} \lra \m{A}_{g-1}$ has been widely studied and it is known that it is generically injective for $g\geq 7$, generically finite of degree $27$ for $g=6$, and dominant with positive dimensional fibres for $g\leq 5$. The ramified case has deserved less attention in the literature. Recently the first author and Pirola have performed a systematic study of the injectivity for $g\geq 2$ (see \cite{PrymVP} and also \cite{miatesi}). They proved that, apart from two sporadic cases, the map is generically injective when the dimension of the source space is strictly lower than the dimension of the target. In the only equidimensional case it was known by \cite{bardcilverra} and \cite{pol1122} that the map is dominant of degree $3$.

Our paper deals with the study of the Prym map in the case of double coverings of elliptic curves ramified at $r$ points. When $r\leq 4$ the generic fibre has positive dimension and for $r=4$ it is completely described in \cite{BarthSurf}. Our main result completes the study of the ramified Prym map. Namely we prove (see Section \ref{sec:notation} for the notation):
\begin{theorem}
\label{theo:main}
 The Prym map
\[
 \prym\colon \m{R}_{1,r} \lra \m{A}_{\frac{r}{2}}^\delta 
\]
is generically injective for $r\geq 6$.
\end{theorem}
In contrast with the previous cases we get the generic injectivity also for $r=6$, when the two moduli spaces have the same dimension. In this case we obtain as a byproduct the unirationality of $\m{R}_{1,6}$ and $\m{A}_{3}^\delta$ (see Corollary \ref{cor:unirationality}). The case of the moduli of abelian threefolds was obtained by Kanev in \cite{HurwitzSpacesKanev}.

In Section \ref{sec:DelCentinaRecillas} we prove the birationality of the Prym map when $r=6$. The rational inverse is explicitly given and its definition is based on a construction by Del Centina and Recillas (see \cite{DelCentinaRecillas}). In Section \ref{sec:proof} we prove the theorem for $r\geq 8$. Starting from the previous case we proceed by induction on $r$ by using a degeneration argument. The proof is a refinement of that given in \cite{PrymVP}. The existence of a non-finite group of automorphisms of the elliptic curve yields technical problems and it is not possible to apply the same method word by word. The key point is Lemma \ref{lemma:stein}, which is a generalization of \cite[Lemma 3.4]{PrymVP}. This result allows us to compute the degree of the Prym map by specializing to certain subvarieties. 

\section{Notation and preliminaries}
\label{sec:notation}

We work over the field $\C$ of complex numbers. We will use the following conventions:
\begin{itemize}
 \item $\m{R}_{g,r}$ is the moduli space of irreducible double coverings of a curve of genus $g\geq 1$ ramified at $r\geq 0$ points. A covering of this type is determined by the base curve $C$, a line bundle $\eta \in \pic^{\frac{r}{2}}(C)$, and a divisor $B\in \vert \eta^2\vert$. The dimension of $\m{R}_{g,r}$ is $3g-3+r$.
 \item  $\m{R}_{g}:=\m{R}_{g,0}$ is the moduli space of \'etale coverings. A point of this space corresponds to a curve $C$ of genus $g$ and a non-trivial point $\sigma\in J(C)$ of order $2$.
 \item To a double covering $D\ra C$ we attach the norm map $J(D)\ra J(C)$. The Prym variety $P(D, C)$ is the identity component of the kernel of this map. Its dimension is $g-1+\frac{r}{2}$ and the polarization $L_D$ of $J(D)$ induces a polarization of type $\delta:= (1, \ldots, 1,2,\ldots, 2)$ on $P(D, C)$, where $2$ is repeated $g$ times. In the \'etale case $L_D$ induces twice a principal polarization on $P(D, C)$.
\item We denote by
\begin{equation*}
\prym\colon \mathcal{R}_{g,r} \ra \mathcal{A}_{g-1+\frac{r}{2}}^\delta
\end{equation*}
the Prym map which associates to a covering its Prym variety.
\item Given a covering $(C,\eta,B)$ we denote by $C_\eta$ the image of the projective map corresponding to $\vert \omega_C\otimes\eta\vert$. We call $C_\eta$ the semicanonical curve.
\item $C_r$ denotes the $r$ symmetric product of $C$.
\end{itemize}

\section{Del Centina--Recillas construction}
\label{sec:DelCentinaRecillas}

In this section we prove the generic injectivity of the Prym map for the case of $6$ branch points. The main tool
we need is a construction provided by Del Centina and Recillas in \cite{DelCentinaRecillas}. They attach to a
generic element $\pa{C, \sigma}\in \m{R}_3$ an \'etale covering of a bielliptic curve. Since we will use the
explicit construction we recall how it works.

Given a generic $\pa{C, \sigma}\in \m{R}_3$, set
\begin{equation*}
\tilde X:= W_2\pa{C}\cap \pa{W_2\pa{C} + \sigma} \subset \pic^2\pa{C};
\end{equation*}
then $\tilde X$ is a smooth connected curve of genus $7$ with $3$ involutions
\begin{equation*}
\begin{aligned}
i_\sigma\colon L &\mapsto L\otimes \sigma,\\
i_{\omega_C}\colon L &\mapsto \omega_C\otimes L^{-1},
\end{aligned}
\end{equation*}
and the composition
\begin{equation*}
 i_{\sigma}':= i_{\omega_C}\circ i_\sigma.
\end{equation*}
The quotient of $\tilde X$ by the $3$ involutions gives a diagram of degree $2$ morphisms as follows
\begin{equation}
 \label{eq:diagrameh}
\begin{gathered}
\xymatrix{
&\tilde X \ar[rd]^{\tilde e} \ar[d]^{\pi} \ar[ld]_{\pi'}&\\
X'\ar[rd]_{\varepsilon'} & X\ar[d]^{\varepsilon} & \tilde E \ar[dl]^{e}\\
&E&
}
\end{gathered}
\end{equation}
where $E$ is the quotient of $\tilde X$ by the $\Z_2 \times \Z_2$ group given by the $3$ involutions. More
precisely
\begin{itemize}
 \item $X= \tilde X/\langle i_\sigma \rangle$, $X'= \tilde X/\langle i'_\sigma \rangle$, $\tilde E= \tilde
X/\langle i_{\omega_C}\rangle$;
\item $X$ and $X'$ have genus $4$ and $\tilde E$ has genus $1$;
\item $\pi, \pi'$, and $e$ are \'etale;
\item $\varepsilon$ and $\varepsilon'$ have the same branch divisor $B$ of degree $6$;
\item $\varepsilon\colon X \ra E$ is determined by $B$ and a line bundle $\eta\in \pic^3\pa{E}$ such that
$B\in\vert\eta^2\vert$, in the same way $\varepsilon'\colon X' \ra E$ is given by $B$ and $\eta'\in\pic^3\pa{E}$
and the non-trivial $2$ order point $\mu:=\eta'\otimes\eta^{-1}$ defines the covering $e\colon \tilde E \ra E$;
\item $\tilde X$ is the fibred product $X\times_{E} \tilde E$, hence $E,\eta,B$, and $\mu$ determine the whole diagram.
\end{itemize}

Motivated by these properties, we consider the moduli space $\m{B}_4$ of the classes of elements $\pa{E,\eta,B,\mu}$ such that $E$ is an elliptic curve and $\eta, B, \mu$ are defined as before. The main theorem in \cite{DelCentinaRecillas} states that the rational map
\begin{align*}
  \rho\colon \m{R}_3 &\dashrightarrow \m{B}_4\\
\pa{C, \sigma} &\mapsto \pa{E,\eta,B,\mu}
\end{align*}
is birational (see \cite[Theorem 2.3.2]{DelCentinaRecillas}). The inverse rational map can be described in the following way: the Prym variety $P(\tilde X , X')$ is isomorphic, as a principally polarized abelian variety, to $J\pa{C}$. Moreover, the pullback map $J\pa{E} \ra J(\tilde X)$ sends the subgroup of $2$ order points to $\Z_2\subset P(\tilde X , X')$ and the corresponding non-trivial generator is mapped to $\sigma$ through the isomorphism. 

To relate this construction to our Prym map we need to consider the moduli space $\m{A}_3^{\delta, 2}$ of polarized abelian threefolds $\pa{P,L}$ of type $\delta:=\pa{1,1,2}$ with a marked non-trivial ($2$ torsion) point $\omega$ such that $t^*_{\omega} L\simeq L$. It is not hard to see that there is a well defined map
\begin{align*}
 \tilde \prym \colon \m{B}_4 &\lra \m{A}_3^{\delta, 2}\\
\pa{E,\eta,B,\mu} &\mapsto (P(X,E),\varepsilon^*\mu).
\end{align*}
We get a commutative diagram
\begin{equation*}
 \begin{gathered}
\xymatrix{
\m{B}_4 \ar[d]_{p_1}\ar[rr]^{\tilde \prym} && \m{A}_3^{\delta, 2}\ar[d]^{p_2}\\
\mathcal{R}_{1,6}\ar[rr]_{\prym} && \mathcal{A}^{\delta}_3
}
\end{gathered}
\end{equation*}
where $p_1$ and $p_2$ are the forgetful map which are \'etale coverings of degree $3$. In order to prove that $\prym$ is birational we will prove that $\tilde \prym$ is.

Consider the rational map
\[
\varphi\colon \m{A}_3^{\delta, 2} \dashrightarrow \mathcal{R}_3
\]
defined as follows.  Given a generic $\pa{P, \omega}\in \m{A}_3^{\delta, 2}$, let $A$ be the quotient of $P$ by the group of order $2$ generated by $\omega$ and let $f\colon P \ra A$ be the projection morphism. Then, there exists a principal polarization $M$ on $A$ such that $f^*M\simeq L$ and there is a smooth curve $C$ of genus $3$ such that $A$ and $J\pa{C}$ are isomorphic as principally polarized abelian varieties. We denote by $\sigma$ the unique non-zero point in the image of $\set{x\in P \, \vert \, t_x^*L\simeq L}$ in $J\pa{C}$ and we set
\[
 \varphi\pa{P, \omega}:=\pa{C,\sigma}.
\]

\begin{theorem}
 The Prym map
\[
 \prym\colon \m{R}_{1,6} \lra \m{A}_{3}^\delta 
\]
is generically injective.
\end{theorem}
\begin{proof}
The statement is an easy consequence of the commutativity of the next diagram since $\rho$ is birational
 \begin{equation*}
 \begin{gathered}
\xymatrix{
&&&\mathcal{R}_3\ar@/_1pc/@{-->}[dlll]_{\rho}\\
\mathcal{B}_4\ar[d]_{p_1}\ar[rr]_{\tilde \prym} && \m{A}_{3}^{\delta,2} \ar[d]^{p_2}\ar@{-->}[ur]_{\varphi}&\\
\mathcal{R}_{1,6}\ar[rr]_{\prym} && \mathcal{A}^{\delta}_3&
}
\end{gathered}
\end{equation*}
To show that $\varphi \circ\tilde \prym \circ \rho= \Id$ we fix a generic $\pa{C, \omega}\in \m{R}_3$. By keeping the notation of diagram \eqref{eq:diagrameh}, we have to prove that
\begin{equation*}
 \varphi\pa{P(X, E), \varepsilon^*\mu}= \pa{C,\sigma}.
\end{equation*}
 It is easy to see that $\pi^*\colon J\pa{X} \ra J(\tilde X)$ restricts to an isogeny 
\[
 \pi^*|_{P(X,E)}\colon P(X,E) \lra P(\tilde X, X') \simeq J(C)
\]
such that $\ker \pi^*|_{P(X,E)}=\langle \varepsilon^*\mu \rangle$. Therefore it is enough to prove that the pullback of the principal polarization of $P(\tilde X, X')$ is the $\pa{1,1,2}$ polarization of $P(X,E)$. This follows easily from the well known fact (see \cite{prymMumford}) that $(\pi^*)^*\fascio{J(\tilde X)}(\Theta_{\tilde X})$ is algebraically equivalent to $\fascio{J(X)}(2\Theta_{X})$.
\end{proof}

It is well known that $\m{R}_3$ and therefore $\m{B}_4$ are rational (see e.g. \cite{DolgRat} and references therein). Thus we obtain:

\begin{cor}
\label{cor:unirationality}
 The moduli spaces $\m{R}_{1,6}$ and $\m{A}_3^\delta$ are unirational.
\end{cor}

\begin{rmk}
 The unirationality of $\m{A}_3^\delta$ is proved in \cite{HurwitzSpacesKanev} by a complete different method.
\end{rmk}

\section{Proof of the theorem}
\label{sec:proof}

The whole section is devoted to the proof of Theorem \ref{theo:main}. We proceed by induction with respect to $r\geq 6$ and even. The initial step $r=6$ has been proved in the previous section, now we give the proof of the induction step. We follow closely the techniques used in \cite{PrymVP} to prove a similar theorem for greater genus. In our situation the existence of a non-finite group of automorphisms of the base curve yields technical problems and it is not possible to apply the same method word by word. We have to refine some tools of that proof, this will be especially clear in Lemma \ref{lemma:stein}. Nevertheless, the general strategy is similar, so we only sketch the proof and we refer to \cite{PrymVP} for further details. First we notice that the codifferential of the Prym map 
\begin{equation*}
d\prym^*\colon \sym^2H^0\pa{E, \eta} \lra H^0\pa{E, \fascio{E}\pa{B}}
\end{equation*}
is generically surjective and therefore $\prym$ is generically finite (see \cite[Proposition 2.2]{PrymVP}). Moreover $\ker d\prym^*$ is the space of quadrics vanishing on the semicanonical curve $E_\eta$ and it is not hard to prove that $E_\eta$ is the intersection of these quadrics (see \cite[Theorem 2.8]{PrymVP}). It follows by a standard argument of infinitesimal variation of Hodge structures that the generic Prym variety determines the base curve $E$ and the line bundle $\eta$. 

Now we use a degeneration argument by keeping fixed the elliptic curve and allowing branch points coincide. We observe that the compactification of the Prym map defined in \cite[Section 3]{PrymVP} also works in the case of genus $1$. So we consider 
\begin{equation*}
\Upsilon:=\set{\pa{\eta, B} \in \pic^{\frac{r}{2}}\pa{E} \times E_r \,\vert \, B\in \abs{\eta^2}},
\end{equation*}
and the partition
\[
\Upsilon=\bigsqcup_{k=1}^r Y_k,
\]
where
\begin{equation*}
Y_{k}:=\set{(\eta, \sum_{i} n_iy_i) \in \Upsilon \,\vert \, \sum_i\pa{n_i-1}=k-1}.
\end{equation*}
We remark that $\Upsilon$ is an \'etale $4$-covering of the symmetric product $E_r$ of $E$ and each $Y_{k}$ maps to the $k$-diagonal of $E_r$. In particular, $Y_1$ is an \'etale covering of the open set of divisors with no multiple points.
The rational map
\begin{equation*}
\begin{split}
\mathcal{T} \colon \Upsilon &\dashrightarrow \mathcal{R}_{1,r}\\
\pa{\eta, B}&\mapsto \pa{E, \eta, B},
\end{split}
\end{equation*}
is regular on $Y_1$ and
\begin{equation}
 \label{eq:fibre}
\mathcal{T}^{-1}\pa{\mathcal{T}\pa{\eta, B}}=X_{\eta,B}\cup X_{i^*\eta,i^*B}, 
\end{equation}
where $i$ is an hyperelliptic involution on $E$ and $X_{\eta,B}:=\set{\pa{t_e^*\eta, t^*_e B}}_{e\in E}$. Thus the generic fibre of $\mathcal{T}$ consists of two disjoint copies of $E$. We emphasize that this is the point in which our situation differs from that considered in \cite{PrymVP}, since in that case the base curve has a finite number of automorphisms.

Let us consider the rational map
\begin{equation*}
  \prym_E\colon  \Upsilon \dashrightarrow \mathcal{A}^{\delta}_{\frac{r}{2}}
\end{equation*} 
that is the composition of $\m{T}$ with the Prym map. Obviously this map is regular on $Y_1$. As in \cite{PrymVP} one can see that $\prym_E$ extends to $Y_2$ once we replace $\mathcal{A}^{\delta}_{\frac{r}{2}}$ with the normalized blowing-up $\bar{\mathcal{A}}^{\delta}_{\frac{r}{2}}$ of its Satake compactification. Namely, the map $\prym_E$ extends to a map
 \begin{equation*}
  \bar \prym_E\colon  \Upsilon \dashrightarrow \bar{\mathcal{A}}^{\delta}_{\frac{r}{2}}
\end{equation*}
whose indeterminacy locus is contained in $\bigsqcup_{k\geq 3}Y_k$. By blowing up in a convenient centre we get a regular map
\[
 \tilde\prym_E \colon  \tilde \Upsilon \lra \bar{\mathcal{A}}^{\delta}_{\frac{r}{2}}.
\]
Given a point $x\in E$ we take
\[
 z:=\pa{\eta'\otimes \fascio{E}\pa{x}, B'+2x}\in Y_2.
\]
The admissible covering (in the sense of \cite[Chapter 3, Section G]{HM98}) corresponding to $z$ is as in Figure \ref{disegno} and the image in $\bar{\mathcal{A}}^{\delta}_{\frac{r}{2}}$ is described by the following data:
\begin{itemize}
 \item the compact Prym variety $P(D,E)$;
 \item the class $\pm[p_x-q_x]$ in the Kummer variety of $P(D,E)$.
\end{itemize}

\begin{figure}
   \begin{center}
    \begin{picture}(200,200)
     \put(-100,-20){\includegraphics[scale=0.4]{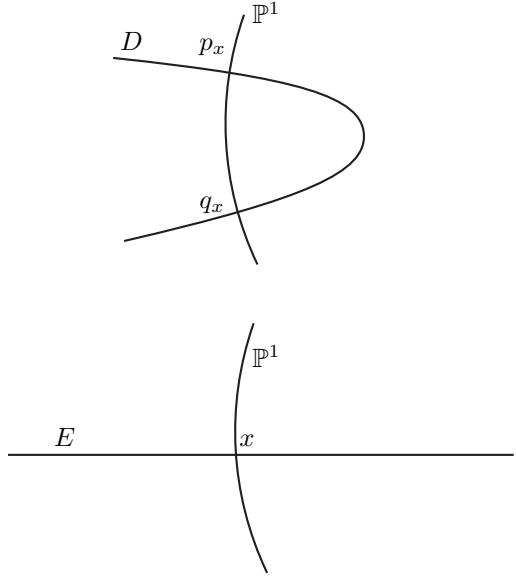}}
     \put(35,40){$E$}
     \put(60,190){$D$}
      \put(110,70){$\mathbb{P}^1$}
     \put(105,40){$x$}
      \put(90,130){$q_x$}
      \put(90,190){$p_x$}
      \put(110,200){$\mathbb{P}^1$}
    \end{picture}
   \end{center}
   \caption{Admissible double covering corresponding to $z$}\label{disegno}
  \end{figure}

Now we want to compute the degree of the Prym map by looking at the behaviour of $\tilde \prym_E$ along $Y_2$. To this end we will need the following generalization of \cite[Lemma 3.4]{PrymVP} (see also \cite[Lemma 2.15]{miatesi}).

\begin{lemma}
\label{lemma:stein}
Let $f\colon X \ra Z$ be a surjective, proper morphism of varieties over an algebraically closed field $\Bbbk$ such that $\dim Z\geq 2$. Consider an integral, locally closed subset $Y$ of $X$ of codimension $1$, not contained in the singular locus of $X$, and set $H:=f\pa{Y}\cap f\pa{Y^c}$. Assume that:
\begin{enumerate}
  \item \label{item:lemmaastratto1B} the codimension of $H$ in $Z$ is at least $2$;
  \item \label{item:lemmaastratto2B} the differential of $f$ is surjective at a generic point of $Y$;
  \item \label{item:lemmaastratto3B} the generic fibre at a point of $f\pa{Y}$ has $n$ connected components.
\end{enumerate}
Then the generic fibre of $f$ has $m\leq n$ connected components.
\end{lemma}
\begin{proof}
 By using Stein factorization theorem (see \cite[Corollaire 4.3.3]{EGA3})) the statement is a straightforward extension of the proof of \cite[Lemma 3.4]{PrymVP}.
\end{proof}

We set $f:=\tilde \prym_E$ and we want to consider $Y=Y_2$. The same proof as in \cite[Section 3]{PrymVP} shows that $Y_2$ satisfies the hypotheses of the lemma. We finish the proof of the theorem by noticing that \eqref{eq:fibre} implies that the generic fibre of $\tilde \prym_E$ has at least $2$ connected components. Due to the lemma we know that they are exactly $2$. Since all the elements of these $2$ fibres identify the same covering of $E$, we are done.

\begin{ack}
 We are very grateful to professor Pirola for his encouragement and for many helpful suggestions. We also thank professor Verra for pointing out the reference \cite{DelCentinaRecillas}.
\end{ack}

\bibliographystyle{alpha}

\bibliography{biblio}

\begin{thebibliography}{Mum74}

\bibitem[Bar87]{BarthSurf}
W.~Barth.
\newblock Abelian surfaces with {$(1,2)$}-polarization.
\newblock In {\em Algebraic geometry, {S}endai, 1985}, volume~10 of {\em Adv.
  Stud. Pure Math.}, pages 41--84. North-Holland, Amsterdam, 1987.

\bibitem[BCV95]{bardcilverra}
F.~Bardelli, C.~Ciliberto, and A.~Verra.
\newblock Curves of minimal genus on a general abelian variety.
\newblock {\em Compositio Math.}, 96(2):115--147, 1995.

\bibitem[DCR89]{DelCentinaRecillas}
A.~Del~Centina and S.~Recillas.
\newblock On a property of the {K}ummer variety and a relation between two
  moduli spaces of curves.
\newblock In {\em Algebraic geometry and complex analysis ({P}\'atzcuaro,
  1987)}, volume 1414 of {\em Lecture Notes in Math.}, pages 28--50. Springer,
  Berlin, 1989.

\bibitem[Dol08]{DolgRat}
I.~V. Dolgachev.
\newblock Rationality of {$\mathcal{R}_2$} and {$\mathcal{R}_3$}.
\newblock {\em Pure Appl. Math. Q.}, 4(2, part 1):501--508, 2008.

\bibitem[Gro63]{EGA3}
A.~Grothendieck.
\newblock \'{E}l\'ements de g\'eom\'etrie alg\'ebrique. {III}. \'{E}tude
  cohomologique des faisceaux coh\'erents. {II}.
\newblock {\em Inst. Hautes \'Etudes Sci. Publ. Math.}, (17):91, 1963.

\bibitem[HM98]{HM98}
J.~Harris and I.~Morrison.
\newblock {\em Moduli of curves}, volume 187 of {\em Graduate Texts in
  Mathematics}.
\newblock Springer-Verlag, New York, 1998.

\bibitem[Kan04]{HurwitzSpacesKanev}
V.~Kanev.
\newblock Hurwitz spaces of triple coverings of elliptic curves and moduli
  spaces of abelian threefolds.
\newblock {\em Ann. Mat. Pura Appl. (4)}, 183(3):333--374, 2004.

\bibitem[Mar11]{miatesi}
V.~Marcucci.
\newblock {\em Curves in Jacobian and Prym varieties}.
\newblock PhD thesis, 2011.

\bibitem[MP10]{PrymVP}
V.~Marcucci and G.~P. Pirola.
\newblock Generic {T}orelli theorem for {P}rym varieties of ramified coverings.
\newblock {S}ubmitted, 2010.
\newblock arXiv:1010.4483v2.

\bibitem[Mum74]{prymMumford}
D.~Mumford.
\newblock Prym varieties. {I}.
\newblock In {\em Contributions to analysis (a collection of papers dedicated
  to {L}ipman {B}ers)}, pages 325--350. Academic Press, New York, 1974.

\bibitem[NR95]{pol1122}
D.~S. Nagaraj and S.~Ramanan.
\newblock Polarisations of type {$(1,2,\cdots,2)$} on abelian varieties.
\newblock {\em Duke Math. J.}, 80(1):157--194, 1995.

\end{thebibliography}
\end{document}